\documentclass[reqno,11pt]{amsart}
\usepackage{amssymb}
\usepackage{mathrsfs}
\usepackage{amsmath,amssymb}
\usepackage{paralist}
\usepackage{graphics} 
\usepackage{epsfig} 
\usepackage[colorlinks=true]{hyperref}
\hypersetup{urlcolor=red, citecolor=blue}
\usepackage[top=1in,bottom=1in,left=1in,right=1in]{geometry}
\usepackage{cite}
\usepackage{amsmath}
\newtheorem{thm}{Theorem}[section]
\newtheorem{cor}[thm]{Corollary}
\newtheorem{lem}[thm]{Lemma}

\theoremstyle{definition}
\newtheorem{defn}{Definition}[section]

\numberwithin{equation}{section}

\begin{document}
\title[Renormalized solutions to a chemotaxis system ]
{Renormalized solutions to a chemotaxis system with consumption of chemoattractant }%
\author[Wang]{Wang Hengling}%
\address{Department of Mathematics, Southeast University, Nanjing 210096, P. R. China}
\email{hlwang@seu.edu.cn}
\author[Li]{Li Yuxiang}%
\address{Department of Mathematics, Southeast University, Nanjing 210096, P. R. China}
\email{lieyx@seu.edu.cn}
\thanks{Supported in part by NSF of China (No. 11671079, 11601127, 11701290) and NSF of Jiangsu Province (No. BK20170896).}

\begin{abstract}
This paper investigates a high-dimensional chemotaxis system with consumption of chemoattractant
 \begin{eqnarray*}
  \left\{\begin{array}{l}
     u_t=\Delta u-\nabla\cdot(u\nabla v),\\
     v_t=\Delta v-uv,\\
  \end{array}\right.
\end{eqnarray*}
under homogeneous boundary conditions of Neumann type, in a bounded convex domain $\Omega\subset\mathbb{R}^n~(n\geq4)$ with smooth boundary. It is proved that if initial data satisfy $u_0\in C^0(\overline{\Omega})$ and $v_0\in W^{1,q}(\Omega)$ for some $q>n$, the model possesses at least one global renormalized solution.
\end{abstract}
\maketitle

\section{Introduction}
In this paper, we consider the global existence of renormalized solutions to the chemotaxis system with consumption of chemoattractant
 \begin{eqnarray}\label{C}
  \left\{\begin{array}{lll}
     u_t=\Delta u-\nabla\cdot(u\nabla v),&{} x\in\Omega,\ t>0,\\
     v_t=\Delta v-uv,&{}x\in\Omega,\ t>0,\\
     \frac{\partial u}{\partial \nu}=\frac{\partial v}{\partial \nu}=0,&{}x\in \partial\Omega,\ t>0,\\
      u(x,0)=u_0(x),~ v(x,0)=v_0(x),&{}x\in \Omega,
  \end{array}\right.
\end{eqnarray}
in a bounded convex domain $\Omega\subset\mathbb{R}^n~(n\geq4)$ with smooth boundary, where the scalar functions $u =u(x, t)$ and $v=v(x, t)$ denote bacterial density and the oxygen concentration, respectively.
$u_0$ and $v_0$ are given functions. $\frac{\partial}{\partial \nu}$ denotes the differentiation with respect to the outward normal derivative on $\partial\Omega$.
Model (\ref{C}) was initially introduced by Keller and Segel \cite{KS1971} to describe the traveling band behavior of chemotactic bacteria, that is, the biased movement of bacteria to the oxygen concentration gradient.
It can be regarded as the `fluid-free' version of the coupled chemotaxis-fluid model which was first presented in \cite{Tuval2005}. Aerobic bacteria such as \emph{Bacillus subtilis} often live in thin fluid layers near solid-air-water contact line, in which the biology of chemotaxis, metabolism, and cell-cell signaling is intimately connected to the physics of buoyancy, diffusion, and mixing \cite{Tuval2005}.

In the last few years, model (\ref{C}) has been studied by some authors. It has been shown by Tao \cite{Tao-JMAA-2011} that (\ref{C}) admits global classical bounded solutions under the assumption that $n\geq2$ and $\|v_0\|_{L^\infty(\Omega)}$ be sufficiently small.
In \cite{Tao&Winkler-JDE-2012}, Tao and Winkler proved that if $n=2$, (\ref{C}) possesses a unique global classical solution which is bounded and satisfies $u(x,t)\rightarrow \bar{u}_0:=\frac{1}{|\Omega|}\int_\Omega u_0$ and $v(x,t)\rightarrow0$
as $t\rightarrow\infty$; in the case $n=3$, for arbitrary large initial data, this problem possesses at least one global weak solution which becomes eventually smooth and also satisfies $(u,v)\rightarrow(\bar{u}_0,0)$ as $t\rightarrow\infty$.
Furthermore, Zhang and Li \cite{Zhang&Li-JMP-2015} obtained that if either $n\leq2$ or $\|v_0\|_{L^\infty(\Omega)}\leq\frac{1}{6(n+1)}, n\geq3$, the global classical solution of (\ref{C}) converges to $(\bar{u}_0,0)$ exponentially in the large time limit.
Similarly, the chemotaxis fluid system has been investigated by some authors, we refer to \cite{Duan&Lorz&Markowich-CPDE-2010,Liu&Lorz-AIHP-2011,Winkler-CPDE-2012,Lorz-M3AS-2010,Winkler-ARMA-2014,Zhang&Li-DCDS-2015} for the further reading.

The concept of renormalized solutions was introduced by Diperna and Lions \cite{Diperna&Lions-CMP-1988,Diperna&Lions-AM-1988,Diperna&Lions-IM-1988}.
In \cite{Fischer-ARMA-2015}, Fischer established the global existence of renormalized solutions to reaction-diffusion systems with entropy-dissipating reactions.
Chen and J\"{u}ngel \cite{Chen-ARXIV-201711} proved the global-in-time existence of renormalized solutions to reaction-cross-diffusion systems for an arbitrary number of variables in bounded domains with no-flux boundary conditions.
For the global existence of renormalized solutions of the Landau equation and Boltzmann equation, see for example \cite{Alexandre-CRASPSIM-1999,Alexandre&Villani-AIHP-2004,Villani-ADE-1996}.
Recently, it was shown in \cite{Winkler-JDE-2018} that the Keller-Segel system with singular sensitivity and signal absorption admits renormalized radial solutions $(u,v)$ which are continuous in $(\overline{\Omega}\setminus\{0\})\times[0,\infty)$ and smooth in $(\overline{\Omega}\setminus \{0\})\times(0,\infty)$, and which solve the corresponding initial-boundary value problem in an appropriate generalized sense.

To the best of our knowledge, for arbitrarily large initial data, whether any kind of solution to (\ref{C}) in high-dimensional exists globally has been an open problem. The difficulty mainly arises from the cross-diffusive term in the first equation when considering the global existence of weak solutions. The known energy estimates are not sufficient to guarantee the boundedness of $u\nabla v$ in $L^s(\Omega\times[0,T])(s>1)$. Therefore, we consider renormalized solutions.

\vskip 3mm
\textbf{Main results.} As usual, we shall assume that the initial data $u_0$ and $v_0$ satisfy
\begin{eqnarray}\label{initialdata}
  \left\{\begin{array}{l}
     u_{0}\in C^0(\overline{\Omega}),\ \ \ u_0>0\ \mathrm{in}\ \overline{\Omega},\\
    v_0\in W^{1,q}(\Omega)\ \ \ \mathrm{for\ some}\ q>n,\ v_0>0\ \mathrm{in}\  \overline{\Omega}.
  \end{array}\right.
\end{eqnarray}
Our main result reads as follows.
\begin{thm}\label{main result}
Let $\Omega\subset\mathbb{R}^n, n\geq4$ be a bounded convex domain with smooth boundary, and suppose that $u_0$ and $v_0$ satisfy (\ref{initialdata}). Then there exists a global renormalized solution of (\ref{C}) in the sense of Definition \ref{defnaaaa} below.
\end{thm}

In order to construct renormalized solutions, we use the notation from \cite{Fischer-ARMA-2015}. Let $\varphi_E: \mathbb{R}_0^+\rightarrow\mathbb{R}_0^+, E\in\mathbb{N}$, be truncation function subject to the following conditions:\\
(E1) Let $\varphi_E\in C^2(\mathbb{R}_0^+)$.\\
(E2) Assume that there exists $K_1>0$ so that $v|\varphi_E''(v)|\leq K_1$ holds for all $E$ and all $v\in\mathbb{R}_0^+$.\\
(E3) Suppose that for every $E$ the set supp$D\varphi_E$ is bounded.\\
(E4) Assume that $\lim_{E\rightarrow\infty}\varphi_E'(v)=1$ holds for all $v\in\mathbb{R}_0^+$.\\
(E5) Suppose that there exists $K_2>0$ such that $|\varphi_E'(v)|\leq K_2$ holds for every $v\in\mathbb{R}_0^+$ and every $E$.\\
(E6) Assume that $\varphi_E(v)=v$ holds for any $v\in\mathbb{R}_0^+$ with $v<E$.\\
(E7) Suppose that we have $\lim_{E\rightarrow\infty}\sup_{|v|\leq K}|\varphi_E''(v)|=0$ for every $K\in\mathbb{R}^+$.

\vskip 3mm

Such truncations $\varphi_E$ satisfying (E1)-(E7) can indeed be constructed. Let $\phi\in C^\infty(\mathbb{R})$ be a smooth nonincreasing function taking values in $[0,1]$ with $\phi\equiv1$ for $x<0$ and $\phi\equiv0$ for $x>1$. Define
\begin{equation}\label{truncation function}
 \varphi_E(v):=v\phi\left(\frac{v-E}{E}\right)+3E\left(1-\phi\bigg(\frac{v-E}{E}\bigg)\right).
\end{equation}
Then one verifies readily that $\varphi_E$ satisfy conditions (E1)-(E7). Note that we shall also use the same family of truncations in the construction of our renormalized solutions below.

The rest of this paper is organized as follows. In Section 2, we introduce a family of regularized problems and give some preliminary properties. Based on an energy-type inequality, a priori estimates are given in Section 3. Section 4 is devoted to showing the global existence of the regularized problems. Finally, we give the proof of the main result in Section 5.

\section{Approximate problems}
According to the idea from \cite{Tao&Winkler-JDE-2012}, we consider the approximate problems
 \begin{eqnarray}\label{regularized problem}
  \left\{\begin{array}{lll}
     u_{\varepsilon t}=\Delta u_{\varepsilon}-\nabla\cdot(u_{\varepsilon}F_\varepsilon'(u_\varepsilon)\nabla v_\varepsilon),&{} x\in\Omega,\ t>0,\\
     v_{\varepsilon t}=\Delta v_\varepsilon-F_\varepsilon(u_\varepsilon)v_\varepsilon,&{}x\in\Omega,\ t>0,\\
     \frac{\partial u_\varepsilon}{\partial \nu}=\frac{\partial v_\varepsilon}{\partial \nu}=0,&{}x\in \partial\Omega,\ t>0,\\
      u_\varepsilon(x,0)=u_0(x),~ v_\varepsilon(x,0)=v_0(x),&{}x\in \Omega,
  \end{array}\right.
\end{eqnarray}
where $\varepsilon\in(0,1)$.

The approximate function ${F_\varepsilon}$ in (\ref{regularized problem}) can be chosen as
\begin{eqnarray*}
{F_\varepsilon}(s):=\frac1\varepsilon\ln(1+\varepsilon s),\ \ \mathrm{for\ all}\ s\geq0.
\end{eqnarray*}
Note that our choice of ${F_\varepsilon}$ ensures that
\begin{equation}
0\leq{F_\varepsilon}'(s)=\frac1{1+\varepsilon s}\leq1,\ \ \mathrm{and}\ \ 0\leq{F_\varepsilon}(s)\leq s\ \ \mathrm{for\ all}\ s\geq0,\label{eq-F-eq1}
\end{equation}
\begin{equation}
s{F_\varepsilon}'(s)=\frac s{1+\varepsilon s}\leq\frac1\varepsilon,\ \ \mathrm{for\ all}\ s\geq0,\label{eq-F-eq2}
\end{equation}
and that
\begin{equation}
{F_\varepsilon}'(s)\nearrow1\ \ \mathrm{and}\ \ {F_\varepsilon}(s)\nearrow s\ \ \mathrm{as}\ \varepsilon\searrow0\ \ \mathrm{for\ all}\ s\geq0.\label{eq-F-eq3}
\end{equation}
All the above approximate problems admit local-in-time smooth solutions:

\begin{lem}\label{lem2.1}
Suppose that $u_0$ and $v_0$ satisfy (\ref{initialdata}), then for any $\varepsilon\in(0,1)$, there exist $T_{max,\varepsilon}\in(0,\infty]$ and a classical solution $(u_\varepsilon, v_\varepsilon)$ of (\ref{regularized problem}) in $\Omega\times(0,T_{max,\varepsilon})$. Moreover, if $T_{max,\varepsilon}<\infty$, then
\begin{equation}
\|{u_\varepsilon}(\cdot,t)\|_{L^\infty(\Omega)}+\|{v_\varepsilon}(\cdot,t)\|_{W^{1,q}(\Omega)}\rightarrow\infty,\ \mathrm{as} ~t\nearrow T_{max,\varepsilon}.\label{criterion}
\end{equation}
\end{lem}
The proof of this lemma is based on well-established methods involving the Banach fixed point theorem, the standard regularity theories of parabolic equations (see \cite{Winkler-CPDE-2012} for instance).

The following estimates of ${u_\varepsilon}$ and ${v_\varepsilon}$ are basic but important in the proof of our result.

\begin{lem}
For each $\varepsilon\in(0,1)$, we have
\begin{equation}
\int_\Omega{u_\varepsilon}(\cdot,t)=\int_\Omega u_0\ \ \ \mathrm{for\ all}\ t\in(0,T_{max,\varepsilon})\label{mass conservation}
\end{equation}
and
\begin{equation}
\|{v_\varepsilon}(\cdot,t)\|_{L^\infty(\Omega)}\leq\| v_0\|_{L^\infty(\Omega)}\ \ \ \mathrm{in}\ \Omega\times(0,T_{max,\varepsilon}).\label{maximum principle}
\end{equation}
\end{lem}

\begin{proof}
Integrating the first equation in (\ref{regularized problem}), we obtain (\ref{mass conservation}).
And an application of the maximum principle to the second equation in (\ref{regularized problem}) gives (\ref{maximum principle}).
\end{proof}

\section{A priori estimates}
This section is devoted to establishing an energy-type inequality which will play a key role in the derivation of further estimates.
\begin{lem}\label{lem3.1}
For each $\varepsilon\in(0,1)$, the solution of (\ref{regularized problem}) satisfies
\begin{equation}\label{energy inequality}
\frac{d}{dt}\left\{\int_\Omega u_\varepsilon\ln u_\varepsilon+2\int_\Omega |\nabla\sqrt{v_\varepsilon}|^2\right\}+\int_\Omega \frac{|\nabla u_\varepsilon|^2}{u_\varepsilon}+\int_\Omega v_\varepsilon|D^2\ln v_\varepsilon|^2+\frac12\int_\Omega F_\varepsilon(u_\varepsilon)\frac{|\nabla v_\varepsilon|^2}{v_\varepsilon}\leq0
\end{equation}
for all $t\in(0, T_{max,\varepsilon})$.
\end{lem}

\begin{proof}
The proof is based on the first two equations in (\ref{regularized problem}) and  integration by parts,
we refer readers to \cite[Lemmas 3.1-3.4]{Winkler-CPDE-2012} for details.
\end{proof}

We next collect some consequences of the above energy inequality which are convenient for our purpose.

\begin{cor}\label{cor3.2}
There exists $C>0$ such that for all $\varepsilon\in(0,1)$ and $T\in(0, T_{max,\varepsilon})$, the solution of (\ref{regularized problem}) satisfies
\begin{equation}
\int_\Omega u_\varepsilon(\cdot,t)|\log u_\varepsilon(\cdot,t)|\leq C,\ \ \mathrm{for\ all}\ t\in(0,T_{max,\varepsilon}),
\end{equation}
\begin{equation}
\int_0^T\int_\Omega \frac{|\nabla u_\varepsilon|^2}{u_\varepsilon}\leq C,
\end{equation}
\begin{equation}
\int_\Omega |\nabla v_\varepsilon(\cdot,t)|^2\leq C,\ \ \mathrm{for\ all}\ t\in(0,T_{max,\varepsilon}),
\end{equation}
\begin{equation}
\int_0^T\int_\Omega  F_\varepsilon(u_\varepsilon)|\nabla v_\varepsilon|^2\leq C.
\end{equation}
\end{cor}

\begin{proof}
Integrating (\ref{energy inequality}), according to the inequality $z\log z\geq-\frac1e$ for $z>0$ and (\ref{maximum principle}), we obtain the desired results.
\end{proof}

\begin{lem}{\rm(\cite[Lemma 4.1]{Winkler-CPDE-2012})}\label{lem3.3}
Suppose that $n\geq1$ and that $\Omega\subset\mathbb{R}^n$ is a bounded domain with smooth boundary. Let $p>1$ and $r\geq1$ be such that
\begin{equation}
p\leq\frac{n}{(n-2)_+}
\end{equation}
and
\begin{equation}
r\leq\frac{2p}{n(p-1)}.
\end{equation}
Then, for all $T>0$ and each $M>0$, there exists $C(T,M)>0$ such that if $\varphi\in L^2((0,T);W^{1,2}(\Omega))$ is nonnegative with
\begin{equation}
\int_\Omega \varphi(\cdot,t)\leq M\ \mathrm{for\ all}\ t\in(0,T),
\end{equation}
then the estimate
\begin{equation}
\int_0^T \|\varphi\|_{L^p(\Omega)}^r dt\leq C(T,M)\cdot\left\{\int_0^T\int_\Omega \frac{|\nabla\varphi|^2}{\varphi}+1\right\}^{\frac{n(p-1)r}{2p}}
\end{equation}
holds.
\end{lem}

In view of (\ref{mass conservation}), the above lemma immediately implies the following.

\begin{cor}\label{3.4}
Suppose that $n\geq4$. Then for all $T\in(0, T_{max,\varepsilon})$, there exists $C>0$ such that for any $\varepsilon\in(0,1)$, the solution of (\ref{regularized problem}) satisfies
\begin{equation}
\int_0^T\int_\Omega u_\varepsilon^{\frac{n+2}{n}}\leq C.
\end{equation}
\end{cor}

\begin{proof}
It is a consequence of Corollary \ref{cor3.2} and of Lemma \ref{lem3.3} applied to $p:=\frac{n+2}{n}$ and $r:=\frac{n+2}{n}$.
\end{proof}

\section{Global solvability for the approximate problems}\label{sectionfour}
Now we are in position to show that the solution of  the approximate problem (\ref{regularized problem}) is actually global in time. The idea of the
proof is based on the argument in \cite{Winkler-CPDE-2012}. Throughout this section,
all constants below possibly depend on $\varepsilon$.

\begin{lem} \label{lem4.1}
For each $\varepsilon\in(0,1)$, we have $T_{max,\varepsilon}=\infty$; that is, the solutions of (\ref{regularized problem}) are global in time.
\end{lem}

\begin{proof} Assume that $T_{max,\varepsilon}<\infty$ for some $\varepsilon\in(0,1)$. Now pick $p>\frac{nq}{n+q}$ ensures that $\frac12-\frac n2(\frac1p-\frac1q)>0$. Moreover,
$$F_\varepsilon(s)\leq \frac{p}{\varepsilon e}(1+\varepsilon s)^{\frac1p}\ \ \mathrm{for\ all}\  s\geq0.$$
This entails that
\begin{equation}
\|F_\varepsilon(u_\varepsilon)v_\varepsilon\|_{L^p(\Omega)}\leq c_1\ \ \mathrm{for\ all}\ t\in(\frac12T_{max,\varepsilon},T_{max,\varepsilon})\label{4.1}
\end{equation}
because of (\ref{mass conservation}) and (\ref{maximum principle}). As a consequence of (\ref{4.1}), the variation-of-constants formula and well-known smoothing estimates  for the Neumann heat semigroup \cite[Lemma 1.3]{Winkler-JDE-2010} yield the estimate
\begin{eqnarray}
  \|\nabla v_\varepsilon(\cdot,t)\|_{L^q(\Omega)}
  &\leq& \|\nabla e^{t\Delta}v_\varepsilon\big(\frac12T_{max,\varepsilon}\big)\|_{L^q(\Omega)}+\int_{\frac12T_{max,\varepsilon}}^t \|\nabla e^{(t-s)\Delta}F_\varepsilon(u_\varepsilon)v_\varepsilon\|_{L^q(\Omega)}ds \nonumber\\
  &\leq& c_2\left(1+\int_{\frac12T_{max,\varepsilon}}^t \left(1+(t-s)^{-\frac12-\frac n2(\frac1p-\frac1q)}\right)\|F_\varepsilon(u_\varepsilon)v_\varepsilon\|_{L^p(\Omega)}ds\right)\nonumber\\
  &\leq &c_3 \ \ \mathrm{for\ all}\  t\in(\frac34T_{max,\varepsilon},T_{max,\varepsilon})\label{4.2}
\end{eqnarray}
with certain positive constants $c_2$ and $c_3$.

We next use (\ref{4.2}) to estimate $\|u_\varepsilon\|_{L^\infty(\Omega)}$. Now taking any $\beta\in\big(\frac{n}{2q},\frac12\big)$ and letting $B$ denote the operator $-\Delta+1$ in $L^q(\Omega)$ with homogeneous Neumann data, we have $D(B^\beta)\hookrightarrow L^\infty(\Omega)$ (see for example \cite{Horstmann&Winkler-JDE-2005}) and hence we find positive constants $c_4,c_5$ and $c_6$ such that
\begin{eqnarray}
  \|u_\varepsilon(\cdot,t)\|_{L^\infty(\Omega)}
  &\leq& \|u_\varepsilon(\cdot,\frac34T_{max,\varepsilon})\|_{L^\infty(\Omega)}+c_4\int_{\frac34T_{max,\varepsilon}}^t \|B^\beta e^{-(t-s)(B-1)}\nabla\cdot\big(u_\varepsilon F_\varepsilon'(u_\varepsilon)\nabla v_\varepsilon\big)\|_{L^q(\Omega)}ds\nonumber \\
  &\leq& c_5\left(1+\int_{\frac34T_{max,\varepsilon}}^t (t-s)^{-\beta-\frac12}\big\|u_\varepsilon F_\varepsilon'(u_\varepsilon)\nabla v_\varepsilon\big\|_{L^q(\Omega)}ds\right)\nonumber \\
  &\leq& c_6 \ \ \mathrm{for\ all}\ t\in(\frac78T_{max,\varepsilon},T_{max,\varepsilon}).
\end{eqnarray}
Combined with (\ref{4.2}), this contradicts (\ref{criterion}) and thereby proves that $T_{max,\varepsilon}=\infty$.
\end{proof}

\section{Existence of renormalized solutions}
Having established the existence of solutions for our approximate problem, we turn to the proof of the existence of renormalized solutions to the original equations (\ref{C}). Before going into detail, let us first give the definition of renormalized solutions.

\begin{defn}\label{defnaaaa}
Suppose that $n\geq1$, that $\Omega\subset\mathbb{R}^n$ is a bounded domain and that $u_0\in L^1(\Omega)$ and $v_0\in L^1(\Omega)$ are nonnegative. Then a pair $(u,v)$ of functions
\begin{eqnarray*}
&&u\in L_{\rm{loc}}^1(\overline{\Omega}\times[0,\infty)),\\
&&v\in L_{\rm{loc}}^\infty(\overline{\Omega}\times[0,\infty)),
\end{eqnarray*}
satisfying $u\geq0$ and $v\geq0$ a.e. in $\Omega\times(0,\infty)$, will be called a global renormalized solution of (\ref{C}) if for all $\xi\in C^\infty([0,\infty))$ with $\xi'\in C_0^\infty([0,\infty))$ we have

\begin{align}\label{renormalized solution} \nonumber
  -\int_0^\infty\int_\Omega \xi(u)\psi_t -\int_\Omega \xi(u_0)\psi(\cdot,0)=
  &-\int_0^\infty\int_\Omega \xi''(u)|\nabla u|^2\psi -\int_0^\infty\int_\Omega \xi'\nabla u\cdot\nabla\psi\\
  &+\int_0^\infty\int_\Omega u\xi''(u)(\nabla u\cdot\nabla v)\psi +\int_0^\infty\int_\Omega u\xi'(u)\nabla v\cdot\nabla\psi
\end{align}
for all $\psi\in C_0^\infty(\overline{\Omega}\times[0,\infty))$, and if moreover the identity

\begin{equation}\label{weak solution}
  \int_0^\infty\int_\Omega v\psi_t+\int_\Omega v_0\psi(\cdot,0)=\int_0^\infty\int_\Omega \nabla v\cdot\nabla\psi+\int_0^\infty\int_\Omega uv\psi
\end{equation}
is valid for any $\psi\in C_0^\infty(\overline{\Omega}\times[0,\infty))$.
\end{defn}

In the first step, we show that a subsequence of the solutions $u_\varepsilon$ to the approximate problems (\ref{regularized problem}) converges to some limit $u$ as $\varepsilon\rightarrow0$.

\begin{lem}\label{limit u}
Consider a sequence $u_\varepsilon$ of solutions to the approximate problems, with $\varepsilon$ converging to zero. Then there exists a subsequence $u_\varepsilon$ (not relabeled) which converges almost everywhere on $\Omega\times[0,\infty)$ to some limit $u\in L_{\rm{loc}}^\infty([0,\infty);L^1(\Omega))$ with $u|\log u|\in L_{\rm{loc}}^\infty([0,\infty);L^1(\Omega))$. Moreover, the convergence $\sqrt{u_\varepsilon}\rightharpoonup\sqrt{u}$ weakly in $L^2([0,T],H^1(\Omega))$ holds for all $T>0$.
\end{lem}

\begin{proof}
Let $\varphi_E$ be as in (\ref{truncation function}). Noting that
$$\nabla[\varphi_E(u_\varepsilon)]=\varphi_E'(u_\varepsilon)\nabla u_\varepsilon$$
and that supp $\varphi_E'(v)$ is a compact subset of $\mathbb{R}_0^+$, by uniform (with respect to $\varepsilon$) boundedness of $\sqrt{u_\varepsilon}$ in $L^2([0,T];H^1(\Omega))$ for every $T>0$ [this is a consequence of Corollary \ref{cor3.2} and (\ref{mass conservation})] we see that $\varphi_E(u_\varepsilon)$ is uniformly bounded with respect to $\varepsilon$ in $L^2([0,T];H^1(\Omega))$ for every fixed $T>0$ and every fixed $E\in\mathbb{N}$.

Let $\psi\in C^\infty(\overline{\Omega}\times[0,\infty))$ be a smooth function. Testing the first equation of (\ref{regularized problem}) by $\psi\varphi_E'(u_\varepsilon)$ and integrating by parts, we have
\begin{align}\label{time derivative} \nonumber
&\int_\Omega \varphi_E(u_\varepsilon(\cdot,T))\psi(\cdot,T)-\int_\Omega \varphi_E(u_0)\psi(\cdot,0)-\int_0^T\int_\Omega \varphi_E(u_\varepsilon)\psi_t \nonumber\\
=&\int_0^T\int_\Omega \frac{d}{dt}\varphi_E(u_\varepsilon)\psi \nonumber\\
=&-\int_0^T\int_\Omega \varphi_E''(u_\varepsilon)|\nabla u_\varepsilon|^2\psi -\int_0^T\int_\Omega \varphi_E'(u_\varepsilon)\nabla u_\varepsilon\cdot\nabla\psi \nonumber\\
  &+\int_0^T\int_\Omega u_\varepsilon F_\varepsilon'(u_\varepsilon)\varphi_E''(u_\varepsilon)(\nabla u_\varepsilon\cdot\nabla v_\varepsilon)\psi \nonumber\\
  &+\int_0^T\int_\Omega u_\varepsilon F_\varepsilon'(u_\varepsilon)\varphi_E'(u_\varepsilon)\nabla v_\varepsilon\cdot\nabla\psi \nonumber\\
=:&I+II+III+IV.
\end{align}
Using the fact that supp $D\varphi_E$ is a compact subset of $\mathbb{R}_0^+$ and the fact that $\sqrt{u_\varepsilon}$ is uniformly (with respect to $\varepsilon$) bounded in $L^2([0,T];H^1(\Omega))$ for every $T>0$, we see that $\frac{d}{dt}\varphi_E(u_\varepsilon)$ is bounded uniformly in $L^1([0,T];(W^{1,\infty}(\Omega))')$ for every $T>0$ and every fixed $E$.

Since $H^1(\Omega)\hookrightarrow\hookrightarrow L^2(\Omega)$, we may combine this with the boundedness of $\left(\frac{d}{dt}\varphi_E(u_\varepsilon)\right)_{\varepsilon\in(0,1)}$ in $L^1([0,T];(W^{1,\infty}(\Omega))')$ to obtain from the Aubin-Lions Lemma (see for example \cite{Lions1969}) that the sequence $\varphi_E(u_\varepsilon)$ is relatively compact in $L^2([0,T];L^2(\Omega))$ for every fixed $T>0$ and fixed $E\in\mathbb{N}$.
By a diagonal sequence argument (we do not relabel the subsequence), we may assume that for every $E\in\mathbb{N}$ the sequence $(\varphi_E(u_\varepsilon))_\varepsilon$ converges almost everywhere to some measurable limit $w_E$.
From the uniform boundedness of $u_\varepsilon|\log u_\varepsilon|$ in $L^\infty([0,T];L^1(\Omega))$ which holds for every fixed $T>0$ [this is also a consequence of Corollary \ref{cor3.2}] we deduce using (E6) that $\varphi_E(u_\varepsilon)|\log\varphi_E(u_\varepsilon)|$ is also bounded uniformly in $L^\infty([0,T];L^1(\Omega))$ for every fixed $T>0$; moreover, the boundedness is also uniform with respect to $E$.
Thus, by Fatou's Lemma we know that $w_E$ is almost everywhere finite and $w_E|\log w_E|$ is bounded uniformly (with respect to $E$) in $L^\infty([0,T];L^1(\Omega))$.

We now prove that the pointwise limit $\lim_{E\rightarrow\infty}w_E$ exists almost everywhere and define a measurable function $u$ with $u|\log u|\in L^\infty([0,T];L^1(\Omega))$.
If for some $(x,t)$ and some $E$ we have $w_E(x,t)=\lim_{\varepsilon\rightarrow0}\varphi_E(u_\varepsilon(x,t))<E$, then $w_{\tilde{E}}(x,t)=w_E(x,t)$ holds for all $\tilde{E}>E$: by our choice of $\varphi_E$ we know that $\varphi_E(v)<E$ implies $\varphi_E(v)=v=\varphi_{\tilde{E}}(v)$.
If we have $w_E(x,t)<E$, then for $\varepsilon$ small enough it holds that $\varphi_E(u_\varepsilon(x,t))<E$ and therefore we get $w_E(x,t)=w_{\tilde{E}}(x,t)$ for $\tilde{E}>E$. Since $w_E$ is bounded uniformly in $L^\infty([0,T];L^1(\Omega))$, the measure of the set of points $(x,t)$ for which $w_E(x,t)\geq E$ holds tends to zero as $E\rightarrow\infty$; thus, the limit $\lim_{E\rightarrow\infty}w_E(x,t)$ exists for almost every $(x,t)\in\Omega\times[0,T]$ and defines a measurable function $u$.
The estimate $u|\log u|\in L^\infty([0,T];L^1(\Omega))$ is a consequence of Fatou's Lemma.

The function $u$ is now the natural candidate for being a renormalized solution of (\ref{C}).

First we notice that (after possibly passing to another subsequence) $u_\varepsilon$ converges almost everywhere to $u$. By uniform boundedness of $u_\varepsilon$ in $L^1(\Omega\times[0,T])$, the measure of the set of points $(x,t)$ with $u_\varepsilon(x,t)\geq E$ tends to zero as $E\rightarrow\infty$, uniformly in $\varepsilon$; thus the measure of the set of points $(x,t)$ for which $\varphi_E(u_\varepsilon(x,t))\neq u_\varepsilon(x,t)$ holds tends to zero as $E\rightarrow\infty$, uniformly in $\varepsilon$. We have for any $\delta>0$
\begin{align*}
 &\mathcal{L}^{n+1}\bigg(\bigg\{(x,t)\in\Omega\times[0,T]:|u_\varepsilon(x,t)-u(x,t)|>\delta\bigg\}\bigg)\\
 \,&\leq \mathcal{L}^{n+1}\bigg(\bigg\{(x,t)\in\Omega\times[0,T]:u_\varepsilon(x,t)\neq\varphi_E(u_\varepsilon)(x,t)\bigg\}\bigg)\\
  &\quad + \mathcal{L}^{n+1}\left(\left\{(x,t)\in\Omega\times[0,T]:|\varphi_E(u_\varepsilon)(x,t)-w_E(x,t)|>\frac\delta2\right\}\right)\\
   & \quad+ \mathcal{L}^{n+1}\left(\left\{(x,t)\in\Omega\times[0,T]:|w_E(x,t)-u(x,t)|>\frac{\delta}{2}\right\}\right),
\end{align*}
where by the previous considerations the first term on the right-hand side converges to zero as $E\rightarrow\infty$, uniformly in $\varepsilon>0$.
The last term tends to zero as $E\rightarrow\infty$ by the definition of $u$; it is independent of $\varepsilon$. The penultimate term converges to zero as $\varepsilon\rightarrow0$ for fixed $E$.
To sum up, we have shown that $u_\varepsilon$ converges to $u$ in measure, which implies convergence almost everywhere for a subsequence.

As $u_\varepsilon$ is bounded uniformly in $L^\infty([0,T];L^1(\Omega))$ for every $T>0$, we deduce that $u_\varepsilon$ converges to $u$ strongly in $L^p([0,T];L^1(\Omega))$ for every $T>0$ and $p\geq1$.
This, in particular, implies convergence of $\sqrt{u_\varepsilon}$ to $\sqrt{u}$ in the sense of distribution, and we obtain that $\sqrt{u}\in L^2([0,T];H^1(\Omega))$ with
\begin{equation*}
  \int_0^T\int_\Omega |\nabla\sqrt{u}|^2\leq\liminf_{\varepsilon\rightarrow0} \int_0^T\int_\Omega |\nabla\sqrt{u_\varepsilon}|^2
\end{equation*}
[the latter $\liminf$ being finite due to Corollary \ref{cor3.2}]. In particular, $\sqrt{u_\varepsilon}$ converges to $\sqrt{u}$ weakly in $L^2([0,T];H^1(\Omega))$ for every $T>0$.
\end{proof}

In the second step, as a preparation for the proof of Theorem \ref{main result}, we show that a subsequence of the solutions $v_\varepsilon$ to the approximate problems (\ref{regularized problem}) converges to some limit $v$ as $\varepsilon\rightarrow0$.

\begin{lem}\label{lem5.2}
Consider a sequence $v_\varepsilon$ of solutions to the approximate problems, with $\varepsilon$ converging to zero. Then there exists a subsequence $v_\varepsilon$ (not relabeled)which converges almost everywhere on $\Omega\times[0,\infty)$ to some limit $v\in L_{\rm{loc}}^\infty(\overline{\Omega}\times[0,\infty))$. Moreover, $v$ satisfies (\ref{weak solution}).
\end{lem}

\begin{proof}
Firstly, we show that $v_\varepsilon$ is strongly precompact in $L^1([0,T];W^{1,1}(\Omega))$. Then since $0\leq F_\varepsilon(u_\varepsilon)\leq u_\varepsilon$, in view of (\ref{maximum principle}) and Corollary \ref{3.4} we can pick positive constant $C$ such that

\begin{equation*}
  \int_0^T\int_\Omega |F_\varepsilon(u_\varepsilon)v_\varepsilon|^{\frac{n+2}{n}}\leq C\int_0^T\int_\Omega u_\varepsilon^{\frac{n+2}{n}}\leq C
\end{equation*}
for all $\varepsilon\in(0,1)$. This shows that $(F_\varepsilon(u_\varepsilon)v_\varepsilon)_{\varepsilon\in(0,1)}$ is bounded in $L^{\frac{n+2}{n}}(\Omega\times(0,T))$, so that standard results on Sobolev regularity for the heat equation \cite{Hieber-CPDE-1997} assert boundedness of both $(v_{\varepsilon t})_{\varepsilon\in(0,1)}$ in $L^{\frac{n+2}{n}}(\Omega\times(0,T))$ and of $(v_\varepsilon)_{\varepsilon\in(0,1)}$ in $L^{\frac{n+2}{n}}((0,T);W^{2,\frac{n+2}{n}}(\Omega))$.
Again by the Aubin-Lions lemma, this shows that $(v_\varepsilon)_{\varepsilon\in(0,1)}$ is relatively compact in $L^1([0,T];W^{1,1}(\Omega))$. It is possible to pick a sequence of numbers $(0,1)\ni\varepsilon_j\searrow0$ such that as $\varepsilon=\varepsilon_j\searrow0$, the solutions $v_\varepsilon$ of (\ref{regularized problem}) satisfy

$$v_\varepsilon\rightarrow v \quad\mathrm{in}\ L_{\rm{loc}}^1(\overline{\Omega}\times[0,\infty)) \ \ \mathrm{and\ a.e.\ in}\ \Omega\times(0,\infty),$$
$$\nabla v_\varepsilon\rightarrow \nabla v \quad\mathrm{in}\ L_{\rm{loc}}^1(\overline{\Omega}\times[0,\infty))\ \ \mathrm{and\ a.e.\ in}\ \Omega\times(0,\infty), $$
for some limit function $v$.
To see that $v$ satisfies (\ref{weak solution}), we fix $\psi\in C_0^\infty(\overline{\Omega}\times[0,\infty))$. Multiplying the second equation in (\ref{regularized problem}) by $\psi$, on integrating by parts we obtain

\begin{equation*}
  \int_0^\infty\int_\Omega v_\varepsilon\psi_t+\int_\Omega v_0\psi(\cdot,0)=\int_0^\infty\int_\Omega \nabla v_\varepsilon\cdot\nabla\psi+\int_0^\infty\int_\Omega F_\varepsilon(u_\varepsilon)v_\varepsilon\psi.
\end{equation*}
Combined with the boundedness of $F_\varepsilon(u_\varepsilon)v_\varepsilon$ in $L^{\frac{n+2}{n}}(\Omega\times(0,T))$, we derive that (\ref{weak solution}) by letting $\varepsilon\rightarrow0$ and thereby completes the proof.
\end{proof}

In the third step of our proof of the existence of renormalized solutions, we show that the "truncations" $\varphi_E(u)$ of the limit $u$, which has been constructed in the first step, satisfy a certain PDE.

\begin{lem}
Let $u$ be the functions constructed in the previous lemma. Let $\varphi_E$ be the functions defined in (\ref{truncation function}). Let $\psi\in C_0^\infty(\overline{\Omega}\times[0,\infty))$. Then $\varphi_E(u)$ satisfies

\begin{align}\label{varphi}
 -&\int_0^\infty\int_\Omega \varphi_E(u)\frac{d}{dt}\psi dxdt-\int_\Omega \varphi_E(u_0)\psi(\cdot,0)dx\nonumber\\
  =&-\int_{\overline{\Omega}\times[0,\infty)}\psi d\mu^E(x,t)\nonumber\\
    &-\int_0^\infty\int_\Omega \varphi_E'(u)\nabla u\cdot\nabla\psi dxdt\nonumber\\
    &+\int_0^\infty\int_\Omega u\varphi_E'(u)\nabla v\cdot\nabla\psi dxdt
 \end{align}
where $\mu^E$ denotes a sequence of signed Radon measures satisfying

\begin{equation}\label{Radon measure}
 \lim_{E\rightarrow\infty}|\mu^E|(\overline{\Omega}\times[0,T))=0
\end{equation}
for all $T>0$.
\end{lem}

\begin{proof}
Let $T>0$ and $\psi\in C_0^\infty(\overline{\Omega}\times[0,T))$. For fixed $E\in\mathbb{N}$ we pass to the limit $\varepsilon\rightarrow0$ in (\ref{time derivative}). Convergence of the left-hand side and of the terms $II, IV$ is immediate by the convergence properties proven in previous lemma and by the fact that supp $D\varphi_E$ is compact.

Two terms whose convergence cannot be ensured are term $I, III$. In order to deal with them, we intend to show that they vanish in the limit $E\rightarrow\infty$. Consider the signed measures

\begin{align}\label{Radon measure definition}
  \mu_\varepsilon^E&:= \left(\varphi_E''(u_\varepsilon)|\nabla u_\varepsilon|^2-u_\varepsilon F_\varepsilon'(u_\varepsilon)\varphi_E''(u_\varepsilon)\nabla u_\varepsilon\cdot\nabla v_\varepsilon\right)dxdt\nonumber \\[5pt]
    &\ = \left(4u_\varepsilon\varphi_E''(u_\varepsilon)|\nabla \sqrt{ u_\varepsilon} |^2-2u_\varepsilon^{\frac32}F_\varepsilon'(u_\varepsilon)\varphi_E''(u_\varepsilon)\nabla\sqrt{u_\varepsilon}\cdot\nabla v_\varepsilon\right)dxdt.
\end{align}
Note that we have
\begin{equation*}
 | \mu_\varepsilon^E|(\overline{\Omega}\times[0,T))\leq C\int_0^T\int_\Omega |\nabla \sqrt{ u_\varepsilon} |^2dxdt+\int_0^T\int_\Omega F_\varepsilon(u_\varepsilon)|\nabla v_\varepsilon|^2dxdt,
\end{equation*}
which follows from the definition of $\mu_\varepsilon^E$ using (E2) and Corollary \ref{cor3.2} as well as Young's inequality. The uniform boundedness of $\sqrt{u_\varepsilon}$ in $L^2([0,T];H^1(\Omega)$ for any $T>0$ implies that after passing to a subsequence we may assume that $\mu_\varepsilon^E$ weak-$\ast$ converges on $\overline{\Omega}\times[0,\infty)$ to some limit $\mu^E$ as $\varepsilon$ tends to $0$.

It remains to prove (\ref{Radon measure}). We now consider the measures
\begin{align*}
  \nu_\varepsilon^K &:= \chi_{\{|u_\varepsilon|\in[K-1,K)\}} |\nabla \sqrt{ u_\varepsilon} |^2 dxdt \\
  \gamma_\varepsilon^K&:= \chi_{\{|u_\varepsilon|\in[K-1,K)}\} F_\varepsilon(u_\varepsilon)|\nabla v_\varepsilon|^2 dxdt
\end{align*}
on $\overline{\Omega}\times[0,\infty)$. Using (E2) and Corollary \ref{cor3.2} we deduce from (\ref{Radon measure definition}) that
\begin{align}\nonumber
  |\mu_\varepsilon^E| (\overline{\Omega}\times[0.T))
  \leq &C\sum_{K=1}^\infty\int_0^T\int_\Omega \chi_{\{|u_\varepsilon|\in[K-1,K)\}}u_\varepsilon|\varphi_E''(u_\varepsilon)| |\nabla \sqrt{ u_\varepsilon}|^2 dxdt\nonumber\\
     &+C\sum_{K=1}^\infty\int_0^T\int_\Omega \chi_{\{|u_\varepsilon|\in[K-1,K)\}}u_\varepsilon|\varphi_E''(u_\varepsilon)|F_\varepsilon(u_\varepsilon)|\nabla v_\varepsilon|^2 dxdt\nonumber\\
  \leq &C \sum_{K=1}^\infty \nu_\varepsilon^K(\overline{\Omega}\times[0,T))\cdot\sup_{|v|\in[K-1,K)} v|\varphi_E''(v)|\nonumber\\
     &+C \sum_{K=1}^\infty \gamma_\varepsilon^K(\overline{\Omega}\times[0,T))\cdot\sup_{|v|\in[K-1,K)} v|\varphi_E''(v)|\nonumber
\end{align}
By (E3), for fixed $E\in\mathbb{N}$ only finitely many terms in the series do not vanish. We may therefore pass to the limit $\varepsilon\rightarrow0$. Using the fact that the measure of open sets is lower semicontinuous with respect to weak-$\ast$ convergence of measures, we obtain, after passing to a subsequence (the passage to a subsequence in particular ensuring that the limits in the last line of the next formula exist),
\begin{align}\label{epsilon} \nonumber
    |\mu^E|(\overline{\Omega}\times[0,T))&\leq\liminf_{\varepsilon\rightarrow0}|\mu_\varepsilon^E|(\overline{\Omega}\times[0,T)) \nonumber\\
   &\leq C\sum_{K=1}^\infty \lim_{\varepsilon\rightarrow0}\nu_\varepsilon^K(\overline{\Omega}\times[0,T))\cdot\sup_{|v|\in[K-1,K)} v|\varphi_E''(v)|\nonumber\\
   &\ \ \ +C\sum_{K=1}^\infty \lim_{\varepsilon\rightarrow0}\gamma_\varepsilon^K(\overline{\Omega}\times[0,T))\cdot\sup_{|v|\in[K-1,K)} v|\varphi_E''(v)|.
\end{align}
However, we have
\begin{align*}
  \sum_{K=1}^\infty\nu_\varepsilon^K(\overline{\Omega}\times[0,T))&=\int_0^T\int_\Omega |\nabla \sqrt{ u_\varepsilon} |^2 dxdt\\
 \sum_{K=1}^\infty\gamma_\varepsilon^K(\overline{\Omega}\times[0,T))&=\int_0^T\int_\Omega F_\varepsilon(u_\varepsilon)|\nabla v_\varepsilon|^2 dxdt.
\end{align*}
As the latter quantities are bounded uniformly with respect to $\varepsilon$, we obtain, using Fatou's lemma (for the counting measure on $\mathbb{N}$; recall that the limits in the next formula actually exist since we have passed to an appropriate subsequence),
\begin{align*}
   \sum_{K=1}^\infty \lim_{\varepsilon\rightarrow0}\nu_\varepsilon^K(\overline{\Omega}\times[0,T))&< \infty \\
  \sum_{K=1}^\infty \lim_{\varepsilon\rightarrow0}\gamma_\varepsilon^K(\overline{\Omega}\times[0,T))&< \infty .
\end{align*}
By dominated convergence applied to the counting measure on $\mathbb{N}$ (which is possible by (E2) and (E7) as well as the previous estimate), we deduce from (\ref{epsilon})
\begin{align*}
\limsup_{E\rightarrow\infty}|\mu^E|(\overline{\Omega}\times[0,T))
&\leq C \sum_{K=1}^\infty \lim_{\varepsilon\rightarrow0}\nu_\varepsilon^K(\overline{\Omega}\times[0,T))\lim_{E\rightarrow\infty}\sup_{|v|\in[K-1,K)} v|\varphi_E''(v)|\\
&\quad+C \sum_{K=1}^\infty \lim_{\varepsilon\rightarrow0}\gamma_\varepsilon^K(\overline{\Omega}\times[0,T))\lim_{E\rightarrow\infty}\sup_{|v|\in[K-1,K)} v|\varphi_E''(v)|\\
&=0.
\end{align*}
This finishes the proof of the lemma.
\end{proof}
We can now prove our main result.

\noindent{\emph{Proof of Theorem \ref{main result}.}}
In order to show that $u$ is a renormalized solution, we apply \cite[Lemma 4]{Fischer-ARMA-2015} to the map $v:=\varphi_E(u)$ in order to approximately identify the weak time derivative of $\xi(\varphi_E(u))$; then we pass to the limit $E\rightarrow\infty$ to deduce the equation for $\xi(u)$.

More precisely, choose some $T>0$ as arbitrary but fixed; we then prove that $u$ is a renormalized solution on $[0,T)$. Let $\xi$ be a smooth function with compactly supported derivatives. Recall that $\varphi_E(u)$ satisfies (\ref{varphi}), we see that in \cite[Lemma 4]{Fischer-ARMA-2015} we need to choose
\begin{align*}
v&= \varphi_E(u), \\
\nu&= -\mu^E, \\
q&= 0, \\
w&= 0,\\
z&= u\varphi_E'(u)\nabla v-\varphi_E'(u)\nabla u.
\end{align*}
Obviously, they satisfy the assumptions of \cite[Lemma 4]{Fischer-ARMA-2015}. Thus, we infer that for any $\psi\in C_0^\infty(\overline{\Omega}\times[0,T))$ the function $\xi(\varphi_E(u))$ must satisfy the estimate
\begin{align*}
   \bigg|&-\int_0^T\int_\Omega \xi(\varphi_E(u))\frac{d}{dt}\psi dxdt-\int_\Omega\xi(\varphi_E(u_0))\psi(\cdot,0)dx  \\
   & \quad-\int_0^T\int_\Omega \xi'(\varphi_E(u))u\varphi_E'(u)\nabla v\cdot\nabla\psi dxdt\\
   &\quad+\int_0^T\int_\Omega \xi'(\varphi_E(u))\varphi_E'(u)\nabla u\cdot\nabla\psi dxdt\\
   &\quad-\int_0^T\int_\Omega \psi\xi''(\varphi_E(u))u\varphi_E'(u)\nabla v\cdot\nabla\varphi_E(u)dxdt\\
   &\quad+\int_0^T\int_\Omega \psi\xi''(\varphi_E(u))\varphi_E'(u)\nabla u\cdot\nabla\varphi_E(u)dxdt\bigg|\\
   &\leq C(\Omega)\|\psi\|_{L^\infty}\sup_v|D\xi(v)|\mu^E|(\overline{\Omega}\times[0,T)).
\end{align*}
To obtain the desired equation for $\xi(u)$, we now pass to the limit $E\rightarrow\infty$. To do so, we use (\ref{truncation function}) as well as (\ref{Radon measure}); note that due to (\ref{Radon measure}), the left-hand side must be zero in the limit, that is, we obtain an exact equation in the limit.
Convergence of the terms in the first line is immediate, as is convergence of the terms in the second and the third line [observe that $\varphi_E(u)$ converges pointwise almost everywhere to $u$ and that the $\varphi_E'$ is bounded by a constant by (E5)].

It remains to deal with the fourth and the fifth line. To show convergence of the two terms, besides the fact $\nabla\sqrt{u}\in L^2([0,T];L^2(\Omega))$ we need the following assertion: there exists a constant $r$ such that for all $E>r$ the estimate $u\geq r$ implies $\xi'(\varphi_E(u(x,t)))=\xi'(u(x,t))=0$ and $\xi''(\varphi_E(u(x,t)))=\xi''(u(x,t))=0$.
Given this assertion, convergence of the remaining terms in the previous formula as $E\rightarrow\infty$ is also immediate since one factor in the integrals will be zero as soon as $u(x,t)$ becomes too large.

To show this assertion, choose $r$ so large that supp $D\xi\subset B_r(0)$. Let $E>r$. Then $u(x,t)\geq r$ implies $\varphi_E(u(x,t))\geq r$ and therefore $\xi'(\varphi_E(u(x,t)))=0$ as well as $\xi''(\varphi_E(u(x,t)))=0$. Combining with Lemma \ref{lem5.2}, we finish the proof of Theorem \ref{main result}. $\hfill\Box$


\vskip 2mm

\begin{thebibliography}{10}

\bibitem{Alexandre-CRASPSIM-1999}
{\sc R.~Alexandre}, {\em Une d\'efinition des solutions renormalis\'ees pour
  l'\'equation de {B}oltzmann sans troncature angulaire}, C. R. Acad. Sci.
  Paris S\'er. I Math., 328 (1999), pp.~987--991.

\bibitem{Alexandre&Villani-AIHP-2004}
{\sc R.~Alexandre and C.~Villani}, {\em On the {L}andau approximation in plasma
  physics}, Ann. Inst. H. Poincar\'e Anal. Non Lin\'eaire, 21 (2004),
  pp.~61--95.

\bibitem{Chen-ARXIV-201711}
{\sc X.~Chen and A.~J¨¹ngel}, {\em Global renormalized solutions to
  reaction-cross-diffusion systems},  (2017).

\bibitem{Diperna&Lions-CMP-1988}
{\sc R.~J. DiPerna and P.-L. Lions}, {\em On the {F}okker-{P}lanck-{B}oltzmann
  equation}, Comm. Math. Phys., 120 (1988), pp.~1--23.

\bibitem{Diperna&Lions-AM-1988}
\leavevmode\vrule height 2pt depth -1.6pt width 23pt, {\em On the {C}auchy
  problem for {B}oltzmann equations: global existence and weak stability}, Ann.
  of Math. (2), 130 (1989), pp.~321--366.

\bibitem{Diperna&Lions-IM-1988}
\leavevmode\vrule height 2pt depth -1.6pt width 23pt, {\em Ordinary
  differential equations, transport theory and {S}obolev spaces}, Invent.
  Math., 98 (1989), pp.~511--547.

\bibitem{Duan&Lorz&Markowich-CPDE-2010}
{\sc R.~Duan, A.~Lorz, and P.~Markowich}, {\em Global solutions to the coupled
  chemotaxis-fluid equations}, Comm. Partial Differential Equations, 35 (2010),
  pp.~1635--1673.

\bibitem{Fischer-ARMA-2015}
{\sc J.~Fischer}, {\em Global existence of renormalized solutions to
  entropy-dissipating reaction-diffusion systems}, Arch. Ration. Mech. Anal.,
  218 (2015), pp.~553--587.

\bibitem{Hieber-CPDE-1997}
{\sc M.~Hieber and J.~Pr\"uss}, {\em Heat kernels and maximal {$L^p$}-{$L^q$}
  estimates for parabolic evolution equations}, Comm. Partial Differential
  Equations, 22 (1997), pp.~1647--1669.

\bibitem{Horstmann&Winkler-JDE-2005}
{\sc D.~Horstmann and M.~Winkler}, {\em Boundedness vs. blow-up in a chemotaxis
  system}, J. Differential Equations, 215 (2005), pp.~52--107.

\bibitem{KS1971}
{\sc E.~F. Keller and L.~A. Segel}, {\em Traveling bands of chemotactic
  bacteria: a theoretical analysis.}, Journal of Theoretical Biology, 30
  (1971), pp.~235--248.

\bibitem{Lions1969}
{\sc J.-L. Lions}, {\em Quelques m\'ethodes de r\'esolution des probl\`emes aux
  limites non lin\'eaires}, Dunod; Gauthier-Villars, Paris, 1969.

\bibitem{Liu&Lorz-AIHP-2011}
{\sc J.-G. Liu and A.~Lorz}, {\em A coupled chemotaxis-fluid model: global
  existence}, Ann. Inst. H. Poincar\'e Anal. Non Lin\'eaire, 28 (2011),
  pp.~643--652.

\bibitem{Lorz-M3AS-2010}
{\sc A.~Lorz}, {\em Coupled chemotaxis fluid model}, Math. Models Methods Appl.
  Sci., 20 (2010), pp.~987--1004.

\bibitem{Tao-JMAA-2011}
{\sc Y.~Tao}, {\em Boundedness in a chemotaxis model with oxygen consumption by
  bacteria}, J. Math. Anal. Appl., 381 (2011), pp.~521--529.

\bibitem{Tao&Winkler-JDE-2012}
{\sc Y.~Tao and M.~Winkler}, {\em Eventual smoothness and stabilization of
  large-data solutions in a three-dimensional chemotaxis system with
  consumption of chemoattractant}, J. Differential Equations, 252 (2012),
  pp.~2520--2543.

\bibitem{Tuval2005}
{\sc I.~Tuval, L.~Cisneros, C.~Dombrowski, C.~W. Wolgemuth, J.~O. Kessler, and
  R.~E. Goldstein}, {\em Bacterial swimming and oxygen transport near contact
  lines}, Proceedings of the National Academy of Sciences of the United States
  of America, 102 (2005), pp.~2277--82.

\bibitem{Villani-ADE-1996}
{\sc C.~Villani}, {\em On the {C}auchy problem for {L}andau equation:
  sequential stability, global existence}, Adv. Differential Equations, 1
  (1996), pp.~793--816.

\bibitem{Winkler-JDE-2010}
{\sc M.~Winkler}, {\em Aggregation vs. global diffusive behavior in the
  higher-dimensional {K}eller-{S}egel model}, J. Differential Equations, 248
  (2010), pp.~2889--2905.

\bibitem{Winkler-CPDE-2012}
\leavevmode\vrule height 2pt depth -1.6pt width 23pt, {\em Global large-data
  solutions in a chemotaxis-({N}avier-){S}tokes system modeling cellular
  swimming in fluid drops}, Comm. Partial Differential Equations, 37 (2012),
  pp.~319--351.

\bibitem{Winkler-ARMA-2014}
\leavevmode\vrule height 2pt depth -1.6pt width 23pt, {\em Stabilization in a
  two-dimensional chemotaxis-{N}avier-{S}tokes system}, Arch. Ration. Mech.
  Anal., 211 (2014), pp.~455--487.

\bibitem{Winkler-JDE-2018}
\leavevmode\vrule height 2pt depth -1.6pt width 23pt, {\em Renormalized radial
  large-data solutions to the higher-dimensional {K}eller-{S}egel system with
  singular sensitivity and signal absorption}, J. Differential Equations, 264
  (2018), pp.~2310--2350.

\bibitem{Zhang&Li-DCDS-2015}
{\sc Q.~Zhang and Y.~Li}, {\em Convergence rates of solutions for a
  two-dimensional chemotaxis-{N}avier-{S}tokes system}, Discrete Contin. Dyn.
  Syst. Ser. B, 20 (2015), pp.~2751--2759.

\bibitem{Zhang&Li-JMP-2015}
\leavevmode\vrule height 2pt depth -1.6pt width 23pt, {\em Stabilization and
  convergence rate in a chemotaxis system with consumption of chemoattractant},
  J. Math. Phys., 56 (2015), pp.~081506, 10.

\end{thebibliography}
\end{document}